\documentclass[11pt,a4paper]{amsart}
\usepackage[top=4cm,bottom=4cm,left=3.5cm,right=3.5cm]{geometry}
\usepackage{subcaption}
\usepackage{tabularx}
\usepackage{amsmath,amssymb,amsthm}
\usepackage[english]{babel}
\usepackage{mathtools}
\usepackage{enumerate}
\usepackage[numbers]{natbib}
\usepackage[hidelinks]{hyperref}
\usepackage{tikz}
\usepackage{comment}
\usepackage{tkz-graph}
\usepackage{tikz-cd}
\tikzset{%
    symbol/.style={%
        ,draw=none
        ,every to/.append style={%
            edge node={node [sloped, allow upside down, auto=false]{$#1$}}}
    }
}
\usepackage{xcolor}

\usepackage{algorithm}
\usepackage[noend]{algpseudocode}

\newcommand{\vN}{\mathbb{N}}
\newcommand{\vZ}{\mathbb{Z}}

\newtheorem{theorem}{Theorem}

\newtheorem{lemma}[theorem]{Lemma}

\newtheorem{corollary}[theorem]{Corollary}
\theoremstyle{definition}

\newtheorem{definition}[theorem]{Definition}
\theoremstyle{remark}
\newtheorem{remark}[theorem]{Remark}
\newtheorem{example}[theorem]{Example}

\numberwithin{equation}{section}

\begin{document}

\title[Local to global principle for expected values]{Local to global principle for expected values}

\author[G. Micheli]{Giacomo Micheli}
\address{Department of Mathematics\\
University of South Florida\\
Tampa, FL 33620, United States of America
}
\email{gmicheli@usf.edu}

\author[S. Schraven]{Severin Schraven}
\address{Institute of Mathematics\\
University of Zurich\\
Winterthurerstrasse 190\\
8057 Zurich, Switzerland
}
\email{severin.schraven@math.uzh.ch}

\author[V. Weger]{Violetta Weger}
\address{Institute of Mathematics\\
University of Zurich\\
Winterthurerstrasse 190\\
8057 Zurich, Switzerland
}
\email{violetta.weger@math.uzh.ch}

\subjclass[2010]{}


\maketitle

\begin{abstract}
This paper constructs a new local to global principle for expected values over free $\vZ$-modules of finite rank. In our strategy we use the same philosophy as Ekedhal's Sieve for densities, later extended and improved by Poonen and Stoll in their local to global principle for densities.
We show that under some additional hypothesis on the system of $p$-adic subsets used in the principle, one can use $p$-adic measures also when one has to compute expected values (and not only densities).
Moreover, we show that our additional hypotheses are sharp, in the sense that explicit counterexamples exist when any of them is missing. In particular, a system of $p$-adic subsets that works in the Poonen and Stoll principle is not guaranteed to work when one is interested in expected values instead of densities.
Finally, we provide both new applications of the method,  and immediate proofs for known results. 
\end{abstract}

\section{Introduction}
\label{sec:introduction}

Let $\vZ$ be the set of integers and $d$ be a positive integer.
 The problem of computing the ``probability'', that a randomly chosen element in $\vZ^d$ has a certain property, has a long history dating back to Ces\'{a}ro  \cite{ce3, ce1}.
  
Since no uniform probability distribution exists over $\mathbb Z$, one introduces the notion of density.
The density of a set $T \subset \mathbb{Z}^d$ is defined to be 
\begin{equation*}
\rho(T) = \lim_{H \rightarrow \infty} \frac{|T \cap [-H,H[^d|}{(2H)^d},
\end{equation*}
if the limit exists. 
Density results over $\vZ^d$ have received a great deal of interest recently \cite{dotti2016eisenstein, ferraguti2018set, ferraguti2016mertens, guo2013probability, bib:shparlinskiEisen, bib:heyman2014shifted, lieb2018uniform, maze2011natural, mertens1874ueber,micheli2017local,   micheli162,  micheli2016density, densitiesunimodular, bib:nymann1972probability,  bib:BS, wang2017smith}.

In \cite[Lemma 20]{poonenAnn} Poonen and Stoll show  that the computation  of  densities of many sets $S\subseteq \vZ^d$ defined by \emph{local} conditions (in the $p$-adic sense) can be reduced to measuring the corresponding subsets of the $p$-adic integers.
This technique is an extension of Ekedahl's Sieve \cite{torsten1991infinite}. 
An even more general result is \cite[Proposition 3.2]{bright2016failures}.
The general philosophy is that, under some reasonable assumptions, one should be able to treat the $p$-adic measures of infinitely many  sets $U_p\subseteq \vZ_p^d$ independently when one looks at the density of the corresponding set $\bigcap_{p\in \mathcal P} U_p^C\cap \vZ^d$  over the integers.

In \cite{primenumbereisenstein} the authors computed the expected number of primes, for which an Eisenstein-polynomial satisfies the criterion of Eisenstein. 
In this paper we prove that this is a special case of a much more general principle that allows to compute expected values from $p$-adic measures  of nicely chosen systems of $p$-adic sets. In fact, in this article we focus on extending the method by Poonen and Stoll to compute expected values of the ``random variable'' that counts how many times an element is expected to be in one of the $U_p$'s. 
First, we show that the hypothesis of Poonen and Stoll is not sufficient to guarantee a local to global principle for expected values (see Example \ref{st_ex}).
This led us to add two additional hypotheses on the system $(U_p)_{p \in \mathcal{P}}$, that appears in the Poonen and Stoll principle, in order to prove Theorem \ref{addendum}, the main theorem of this article.
The additional hypotheses we require are sharp and necessary, as examples \ref{ex:nocond1} and \ref{ex:nocond2} show.

With this extension to the local to global principle one can for example compute for 
non-coprime $m$-tuples of integers how many prime factors they have in common on average, or for rectangular non-unimodular matrices how many primes divide all the basic minors 
on average, and of course also the result of \cite{primenumbereisenstein} follows directly. 

The paper is organized as follows: in Section \ref{sec2} we will recall the local to global principle by Poonen and Stoll. 
 In Section \ref{sec3} we will introduce the definition of expected value 
 of a system of $(U_\nu)_{\nu \in M_\mathbb{Q}}$ and then state and prove Theorem \ref{addendum}, the main theorem of this paper.
 
  In Section \ref{sec4} we give some applications of our main theorem that allow very fast computations of expected values over $\vZ^d$ 
   (see for example \cite{primenumbereisenstein} compared with Corollary \ref{cor:eisenstein}).

\section{Preliminaries}\label{sec2}

\begin{definition}
Let $d  $ be a positive integer. The \emph{density} of a set $T \subset \mathbb{Z}^d$ is defined to be 
\begin{equation*}
\rho(T) = \lim_{H \rightarrow \infty} \frac{|T \cap [-H,H[^d|}{(2H)^d},
\end{equation*}
if the limit exists. Then one defines the upper density $\bar{\rho}$ and the lower density $\underline{\rho}$ equivalently with the $\limsup$ and the $\liminf$ respectively. 
\end{definition}

For convenience, let us restate here the lemma of Poonen and Stoll in  \cite[Lemma 20]{poonenAnn}. 
If $S$ is a set, then we denote by $2^S$ its powerset and  by $S^C$ its complement.  Let $\mathcal{P}$ be the set of primes and $M_{\mathbb{Q}}= \{\infty \} \cup \mathcal{P}$ be the set of all places of $\mathbb{Q}$, where we denote by $\infty$ the unique archimedean place of $\mathbb{Q}$. Let $\mu_{\infty}$ denote the Lebesgue measure on $\mathbb{R}^d$ and $\mu_p$ the normalized Haar measure on $\mathbb{Z}_p^d$.  For $T$ a subset of a metric space, let us denote by $\partial (T)$ its boundary, by $\overline{T} $ its closure and by $T^\circ$ its interior.
By $\mathbb{R}_{\geq0}$ we denote the non-negative reals.  A minor of a matrix $A$ is called  basic, if it is the nonzero determinant of a square submatrix of $A$ of maximal order.
\begin{theorem}[\text{\cite[Lemma 20]{poonenAnn}}]\label{poonen}
Let $d$ be a positive integer. Let $U_{\infty} \subset \mathbb{R}^d$, such that $\mathbb{R}_{\geq0} \cdot U_{\infty} = U_{\infty}$ and $\mu_{\infty}(\partial(U_{\infty}))=0.$ Let $s_{\infty}= \frac{1}{2^d}\mu_{\infty}(U_{\infty} \cap [-1,1]^d)$. 
For each prime $p$, let $U_p \subset \mathbb{Z}_p^d$, such that $\mu_p(\partial(U_p)) =0$ and define $s_p = \mu_p(U_p)$. 
Define the following map 
\begin{eqnarray*}
P: \mathbb{Z}^d   &\rightarrow &  2^{M_{\mathbb{Q}}}, \\
a  &\mapsto & \left\{ \nu \in M_{\mathbb{Q}} \mid a \in U_{\nu} \right\}.
\end{eqnarray*}
If the following is satisfied:
\begin{equation} \label{densitycond}
\lim_{M \rightarrow \infty} \bar{\rho}\left( \left\{ a \in \mathbb{Z}^d \mid a \in U_p \  \text{for some prime} \ p > M \right\} \right)=0,
\end{equation}
then:
\begin{itemize}
\item[i)] $\sum\limits_{\nu \in M_{\mathbb{Q}}} s_{\nu}$ converges.
\item[ii)] For $\mathcal{S} \subset 2^{M_{\mathbb{Q}}},$  $\rho(P^{-1}(\mathcal{S}))$ exists, and defines a measure on $2^{M_{\mathbb{Q}}}$.
\item[iii)] For each finite set $S \in 2^{M_{\mathbb{Q}}}$, we have that
\begin{equation*}
\rho(P^{-1}(\{S\})) = \prod_{\nu \in S} s_{\nu} \prod_{\nu \not\in S} (1-s_{\nu}), 
\end{equation*}
and if $\mathcal{S}$ consists of infinite subsets of $2^{M_{\mathbb{Q}}}$, then $\rho(P^{-1}(\mathcal{S}))=0.$
\end{itemize}
\end{theorem}

To show  that \eqref{densitycond} is satisfied, one can often apply the following useful lemma, that can be deduced from the result in \cite{torsten1991infinite}.

\begin{lemma}[\text{\cite[Lemma 2]{bib:loctoglob}}]\label{showdens}
Let $d$ and $M$ be positive integers. Let $f,g \in \mathbb{Z}[x_1, \ldots, x_d]$ be relatively prime. Define
\begin{equation*}
S_M(f,g) = \left\{ a \in \mathbb{Z}^d \mid  f(a) \equiv g(a) \equiv 0 \mod p \ \text{for some prime} \ p > M \right\},
\end{equation*}
then
\begin{equation*}
\lim_{M \rightarrow \infty} \bar{\rho}(S_M(f,g)) = 0. 
\end{equation*}
\end{lemma}

\section{The local to global principle for expected values}\label{sec3}

Observe, that  in Theorem \ref{poonen} one could always choose the finite set  $S$ to be the empty set, which for our purpose will be convenient.

\begin{corollary}\label{emptyset}
For all  $\nu \in M_{\mathbb{Q}}$ let $U_{\nu}$ be chosen as in Theorem \ref{poonen}, corresponding to a finite set $  S \in 2^{M_{\mathbb{Q}}}$. Let us define
\begin{equation*}
U'_{\nu} = \begin{cases}  U_{\nu}^C & \nu \in S, \\
U_{\nu} & \nu \not\in S,
\end{cases}
\end{equation*} and hence
\begin{equation*}
s'_{\nu} = \begin{cases}  1-s_{\nu} & \nu \in S, \\
s_{\nu} & \nu \not\in S,
\end{cases}
\end{equation*} 
and define
\begin{eqnarray*}
P': \mathbb{Z}^d   &\rightarrow &  2^{M_{\mathbb{Q}}}, \\
a  &\mapsto & \left\{ \nu \in M_{\mathbb{Q}} \mid a \in U'_{\nu} \right\}.
\end{eqnarray*}
Then we get
\begin{itemize}
\item[i)] $\sum\limits_{\nu \in M_{\mathbb{Q}}} s'_{\nu}$ converges.
\item[ii)] For $\mathcal{S} \subset 2^{M_{\mathbb{Q}}},$  $\rho(P'^{-1}( \mathcal{S}  ))$ exists and defines a measure on $2^{M_{\mathbb{Q}}}$.
\item[iii)] $\rho(P'^{-1}(\{ \emptyset \} )) = \prod\limits_{\nu \in M_{\mathbb{Q}}} (1-s'_{\nu}) = \rho(P^{-1}(\{ S \} ))$,
where $P$ is the map as in Theorem \ref{poonen}.
\end{itemize}
\end{corollary}
The proof is straightforward and thus we omit the details.

For a fixed $\nu \in M_\mathbb{Q}$ and $U_\nu $ as in Theorem \ref{poonen}, the density of $U_\nu \cap \mathbb{Z}^d$  can be   computed as follows.

\begin{corollary}\label{denofU}
Let $\nu \in M_{\mathbb{Q}}$ and $U_{\nu}$ be chosen as in Theorem \ref{poonen}, then 
$$\rho( U_{\nu} \cap \mathbb{Z}^d) = \mu_{\nu}(U_{\nu})= s_{\nu}.$$
\end{corollary}

\begin{proof}
We set
\begin{equation*}
U'_{\nu'}= \begin{cases} U_{\nu} & \nu' = \nu, \\ \emptyset & \nu' \neq \nu,
\end{cases}
\end{equation*} and let 
\begin{eqnarray*}
P': \mathbb{Z}^d   &\rightarrow &  2^{M_{\mathbb{Q}}}, \\
a  &\mapsto & \left\{ \nu \in M_{\mathbb{Q}} \mid a \in U'_{\nu} \right\}.
\end{eqnarray*}
Then  by Theorem \ref{poonen} we have $\rho(U_\nu \cap \mathbb{Z}^d) =\rho(P'^{-1}(\{\nu\}))=s_\nu$.
\end{proof}

We observe, that the elements $A \in \mathbb{Z}^d$, which are in $U_\nu$  for infinitely many $\nu \in M_{\mathbb{Q}}$ have density zero.
\begin{lemma}
Let $(U_\nu)_{\nu\in M_\mathbb{Q}}$ be as in Theorem \ref{poonen}. 
 Then we have that
$$ \rho(\{A \in \mathbb{Z}^d \mid A \in U_{\nu} \ \text{for infinitely many} \ \nu \in M_{\mathbb{Q}}\})=0.$$
\end{lemma}

\begin{proof}
Since $U_\nu$ were chosen as in Theorem \ref{poonen}, condition \eqref{densitycond} holds, i.e.,
\begin{equation*} 
\lim\limits_{M \rightarrow \infty} \bar{\rho}\left( \left\{ A \in \mathbb{Z}^d \mid A \in U_p \  \text{for some prime} \ p > M \right\} \right)=0.
\end{equation*}
Let us call
\begin{align*}
C_M & = \left\{ A \in \mathbb{Z}^d \mid A \in U_p \  \text{for some prime} \ p > M \right\}, \\
I &= \left\{ A \in \mathbb{Z}^d \mid A \in U_{\nu} \ \text{for infinitely many} \ \nu \in M_{\mathbb{Q}} \right\}.
\end{align*}
Clearly $I \subset C_M$ for all $M \in \mathbb{N}$, hence $$ \overline{\rho}(I) \leq   \lim\limits_{M \to \infty} \overline{\rho}(C_M) = 0$$
and thus $\rho(I)=0$.
\end{proof}

Over $\mathbb{Z}^d$ one can  give a definition of \emph{mean} or \emph{expected value} (see for example \cite{primenumbereisenstein}), as we will now explain.
Observe, that an event with ``probability'' zero should not have any influence on the expected value, this legitimates that over  $\mathbb{Z}^d$ we will exclude the elements $A \in \mathbb{Z}^d$, which are in infinitely many $U_{\nu}$, namely $A \in I$. 
Let us define $[-H,H[^d_I= ([-H,H[^d \cap \mathbb{Z}^d) \setminus I$.

\begin{definition}\label{meandef}
Let  $H$ and $d$  be positive integers and assume that $(U_\nu)_{\nu \in M_\mathbb{Q}}$ satisfy the assumptions of 
 Theorem \ref{poonen}, then we define \textit{the expected value
 of the system} $( U_\nu)_{\nu \in M_\mathbb{Q}}$ to be
\begin{equation*}
\mu  = \lim\limits_{H \to \infty} \displaystyle{\frac{\sum\limits_{A \in [-H,H[^d_I  } \mid \{ \nu \in M_{\mathbb{Q}} \mid A \in U_{\nu} \} \mid }{  (2H)^d} },
\end{equation*}
if it exists.
\end{definition}
This limit essentially gives the expected value of the number of places $\nu$, such that a ``random'' element in $\mathbb{Z}^d$ is in $U_\nu$.

\begin{remark}
The reader should notice that the mean of the system $(U_\nu)_{\nu \in M_\mathbb{Q}}$
should be thought as the ``expected value'' of the function $a\mapsto |P(a)|$, where $P$ is the map of Theorem \ref{poonen}.
\end{remark}.

\begin{definition}
For a set $T$, for which we can compute its density via the local to global principle as in Theorem \ref{poonen}, we say that a \emph{system $(U_\nu)_{\nu \in M_\mathbb{Q}}$ corresponds to $T$}, if $T^C= P^{-1}(\{\emptyset\})$.
\end{definition}

Observe that we can restrict Definition  \ref{meandef} to subsets of $ [-H,H[^d_I$, i.e.,
we define the expected value of the system $(U_\nu)_{\nu \in M_\mathbb{Q}}$ restricted to $ T \subset[-H,H[^d_I$ to be 
\begin{equation*}
\mu_T  = \lim\limits_{H \to \infty} \displaystyle{\frac{\sum\limits_{A \in [-H,H[^d_I\cap T  } \mid \{ \nu \in M_{\mathbb{Q}} \mid A \in U_{\nu} \} \mid }{ \mid  [-H,H[^d_I \cap T \mid} },
\end{equation*}
if it exists.
Note, that this is analogous to the conditional expected value.

\begin{remark}
If $(U_\nu)_{\nu\in M_{\mathbb{Q}}}$ corresponds to $T$, then the expected value of the system $(U_\nu)_{\nu\in M_\mathbb{Q}}$ restricted to $T$
should be thought as the ``expected value'' of the function $a\mapsto |P(a)|$, where $P$ is the map of Theorem \ref{poonen}, when restricted to the elements $a$ of $\vZ^d$ such that $|P(a)|\geq 1$.
\end{remark}

One can easily pass from $\mu$ to $\mu_T$ and viceversa:

\begin{lemma}\label{passtoT}
If the density of $T$ exists and is nonzero and $T$ is such that $T^C \subseteq P^{-1}(\{ \emptyset\})$, then  $\mu$ exists iff $\mu_T$ exists. Furthermore, in that case it holds that $\mu=\mu_T\rho(T)$.
\end{lemma}

\begin{proof}

We observe that
\begin{align*} \mu_{T} &= \lim\limits_{H \to \infty} \displaystyle{\frac{\sum\limits_{A \in [-H,H[^d_I \cap T} \mid \{ \nu \in M_{\mathbb{Q}} \mid A \in U_{\nu} \} \mid }{ \mid [-H,H[^d_I \cap T \mid} } \\
& =  \lim\limits_{H \to \infty} \displaystyle{\frac{\sum\limits_{A \in [-H,H[^d_I \cap T} \mid \{ \nu \in M_{\mathbb{Q}} \mid A \in U_{\nu} \} \mid }{ (2H)^d }} \displaystyle{\frac{(2H)^d}{\mid [-H,H[^d_I \cap T \mid}} \\ 
&=\lim\limits_{H \to \infty} \displaystyle{\frac{\sum\limits_{A \in [-H,H[^d_I \cap T} \mid \{ \nu \in M_{\mathbb{Q}} \mid A \in U_{\nu} \} \mid }{ (2H)^d }} \displaystyle{\frac{1}{\rho(T)}}.
\end{align*}

Let us define 
 $$\tau(A,\nu) = \begin{cases} 1 & A \in U_{\nu}, \\ 0 & \text{else.}
 \end{cases}$$

Note that one can write $[-H,H[^d_I \cap T$ as $[-H,H[^d_I \setminus ([-H,H[^d_I\cap T^C)$, doing so we observe that we can ignore  $T$: 
\begin{align*} \mu_{T}\rho(T) &= \lim\limits_{H \to \infty} \displaystyle{\frac{\sum\limits_{A \in [-H,H[^d_I \cap T} \sum\limits_{\nu \in M_{\mathbb{Q}}} \tau(A,\nu)}{ (2H)^d }} \\
&= \lim\limits_{H \to \infty} \displaystyle{\frac{\sum\limits_{A \in [-H,H[^d_I } \sum\limits_{\nu \in M_{\mathbb{Q}}} \tau(A,\nu) - \sum\limits_{A \in [-H,H[^d_I \cap T^c} \sum\limits_{\nu \in M_{\mathbb{Q}}} \tau(A,\nu)}{ (2H)^d }}. 
\end{align*}
Since $T^C \subseteq P^{-1}(\{\emptyset\})$,  it holds that $\tau(A,\nu)=0$ for all $A \in T^C$ and hence we are left with $\mu$.
\end{proof}

In the applications of this paper one usually chooses $T^C  =P^{-1}(\{ \emptyset \})$, and computes the expected value restricted to $ T$. This is a natural choice, since  by the definition of $T^C$ it holds that none of its elements lie in any of the $U_\nu$, thus we are only considering the subset $T$, where nonzero values are added to the expected value. \\
%

One would now expect that, if the $p$-adic measures of the $U_p$'s of Theorem \ref{poonen} were essentially behaving like probabilities, one would have that the mean of the system $(U_p)_{p\in \mathcal P}$ (as as defined in Definition \ref{meandef}) would be equal to $\sum_{p\in \mathcal{P}} s_p$, since the density of $U_p\cap \vZ^d$ is equal to $s_p$ (for example, this always happens when one has $U_p=\emptyset$ for all but finitely many $U_p$'s). Note that this would also be the result if we could simply move the limit inside the series. This is not the case: in fact,  Condition \eqref{densitycond} of Theorem \ref{poonen}, is not enough to ensure the existence of the mean as in the natural Definition \ref{meandef}. The next example shows a case where Condition \eqref{densitycond} is verified, but the mean does not exist.

\begin{example} \label{st_ex}
	We set $U_\infty = \emptyset$ and for all $j\in \mathbb{N}$ with $2^n \leq j < 2^{n+1}$ we define $U_{p_j}=\{ 2^n\}$, where $p_j$ denotes the $j$th prime number. 	As  $\vZ_p$ with the $p$-adic metric is a metric space, we get that $U_{p_j}$ is a closed set and thus Borel-measurable, of measure zero. 
	Furthermore, every ball in the $p$-adic metric contains infinitely many elements, which implies 
	$\partial (U_{p_j}) = \overline{U}_{p_j} \setminus U_{p_j}^\circ = U_{p_j} \setminus \emptyset = U_{p_j}$. 
	As the $p$-adic Haar measure is invariant under translation and normalized, we get that all finite sets are null sets. In particular, we get $\mu_{p_j}(\partial (U_{p_j}))=0$. Furthermore, we have 
	\begin{align*}
	 \overline{\rho}\left( \bigcup_{p\in \mathcal{P} : p>M} U_p\right)
	\leq  \overline{\rho}\left( \bigcup_{\nu \in M_\mathbb{Q}} U_\nu \right) = \overline{\rho}(\{2^n \ \mid \ n\in \mathbb{N} \}) =0.
	\end{align*}
	Hence, Condition \eqref{densitycond} is satisfied, even without taking the limit in $M$. 
	Let $n\in \mathbb{N}$, then  for $2^n \leq H < 2^{n+1}$ we have
	\begin{align*}
	\{ \nu \in M_\mathbb{Q} \ \mid \  [-H,H[ \cap U_\nu \neq \emptyset \} = \{ p_1, p_2, \dots, p_{2^{n+1} -1} \} .
	\end{align*}
	Thus, we have for all $n\in \mathbb{N}$
	\begin{align*}
	\sum_{A \in [-2^{n+1}, 2^{n+1}[_I} \frac{\vert \{ \nu \in M_\mathbb{Q} \ : \ A \in U_\nu \} \vert }{2\cdot 2^{n+1}}
	= \frac{2^{n+1}-1}{2\cdot 2^{n+1}} \stackrel{n\rightarrow \infty}{\longrightarrow} \frac{1}{2}
	\end{align*}

	and
	\begin{align*}
	\sum_{A \in [-(2^{n+1}+1), 2^{n+1}+1[_I} \frac{\vert \{ \nu \in M_\mathbb{Q} \ : \ A \in U_\nu \} \vert }{2(2^{n+1}+1)}
	= \frac{2^{n+2}-1}{2(2^{n+1}+1)} \stackrel{n\rightarrow \infty}{\longrightarrow} 1.
	\end{align*}
	Hence, the expected value of the system $(U_\nu)_{\nu \in M_\mathbb{Q}}$ does not exist, even though it satisfies all conditions of Theorem \ref{poonen}.
\end{example}



Now we 
state the main theorem, which is  a local to global principle for expected values, extending the results in \cite[Lemma 20]{poonenAnn}.

\begin{theorem}\label{addendum}
Let $H$ and $d$ be   positive integers. 
Let $U_{\infty} \subset \mathbb{R}^d$, such that $\mathbb{R}_{\geq0} \cdot U_{\infty} = U_{\infty}$ and $\mu_{\infty}(\partial(U_{\infty}))=0.$ Let $s_{\infty}= \frac{1}{2^d}\mu_{\infty}(U_{\infty} \cap [-1,1]^d)$. 
For each prime $p$, let $U_p \subset \mathbb{Z}_p^d$, such that $\mu_p(\partial(U_p)) =0$ and define $s_p = \mu_p(U_p)$. 
Define the following map 
\begin{eqnarray*}
P: \mathbb{Z}^d   &\rightarrow &  2^{M_{\mathbb{Q}}}, \\
a  &\mapsto & \left\{ \nu \in M_{\mathbb{Q}} \mid a \in U_{\nu} \right\}.
\end{eqnarray*}
If \eqref{densitycond} is satisfied and for some $\alpha \in [0, \infty)$ there exists an absolute constant $c \in \mathbb{Z}$, such that for all $H\geq 1$ and for all $A \in [-H,H[^d_I$ one has that
\begin{equation}\label{newcond}
\left\vert \left\{ p\in \mathcal{P}  \mid p > H^\alpha, A \in U_p \cap [-H, H[_I^d  \right\} \right\vert <c
\end{equation} 
and and that there exists a sequence  $(v_p)_{p\in \mathcal{P}}$, such that for all $p<H^\alpha$ 
one has that
\begin{eqnarray}
\mid U_p \cap [-H,H[^d_I \mid & \leq  v_p(2H)^d, 
\label{newcond2} \\
\sum_{p \in \mathcal{P}} v_p & \text{converges}, \label{newcond3}
\end{eqnarray}
then it follows that the mean 
of the system $( U_\nu)_{\nu \in M_{\mathbb{Q}}}$, exists and is given by: \begin{align*}
\mu  &=  \sum\limits_{\nu \in M_{\mathbb{Q}}} s_{\nu}. \\
\end{align*}

\end{theorem}

\begin{remark}
In a nuthshell, the additional conditions give control on the number of $U_p$ for which an element in $[-H,H[_I^d$ can live in: Condition \ref{newcond} gives control in the large $p$ regime (and small $H$), and Condition \ref{newcond2} (with \ref{newcond3}) gives control in the small $p$ regime (and large $H$). None of the conditions can be removed, as we will show later with counterexamples in each case.
\end{remark}

 \begin{proof} 

	For $H,M>0$ we split
	\begin{align*}
	\sum\limits_{A \in [-H,H[^d_I  } \frac{\mid \{ \nu \in M_{\mathbb{Q}} \mid A \in U_{\nu} \} \mid }{ (2H)^d}
	= s_1(H,M) + s_2(H,M) + s_3(H,M),
	\end{align*}
	where
	\begin{align*}
	s_1(H,M) &= \sum\limits_{A \in [-H,H[^d_I  } \frac{\mid \{ p \in \mathcal{P} \mid M\leq H^\alpha <p, A \in U_p \} \mid }{ (2H)^d}, \\
	s_2(H,M) &= \sum\limits_{A \in [-H,H[^d_I  } \frac{\mid \{ p \in \mathcal{P} \mid M<p<H^\alpha, A \in U_p \} \mid }{ (2H)^d}, \\
	s_3(H,M) &= \sum\limits_{A \in [-H,H[^d_I  } \frac{\mid \{ \nu \in \mathcal{P} \mid \nu =\infty \text{ or } \nu\leq M, A \in U_{\nu} \} \mid }{ (2H)^d}.
	\end{align*}
	We are going to show that for $j\in \{1,2\}$ we have 
	$$\limsup_{M \to \infty} \limsup_{H \to \infty} \vert s_j(H,M) \vert=0 \quad \text{and} \quad \lim_{M \rightarrow \infty} \lim_{H \rightarrow \infty} s_3(H,M) = \sum_{\nu_{\nu \in M_\mathbb{Q}}} s_\nu,$$
	which readily implies that 
	\begin{align*}
	\mu  = \lim\limits_{H \to \infty} \displaystyle{\frac{\sum\limits_{A \in [-H,H[^d_I  } \mid \{ \nu \in M_{\mathbb{Q}} \mid A \in U_{\nu} \} \mid }{(2H)^d}} 
	\end{align*}
	exists and that
	\begin{align*}
	\mu = \sum_{\nu \in M_{\mathbb{Q}}} s_\nu.
	\end{align*}
	First we consider the case $\alpha \neq 0$.
	Let us define for $H>0$ and $A\in \mathbb{Z}^d$
	\begin{equation} \label{def lAH}
	\mid \left\{ p\in \mathcal{P}  \mid p > H^\alpha, A \in U_p \cap [-H, H[_I^d  \right\} \mid = \ell_{A,H}.
	\end{equation}
	
	Notice that thanks to condition  \eqref{newcond} there exists a constant $c>0$ independent of $A$ and $H$ such that
	\begin{align*}
	\ell_{A,H}< c.
	\end{align*}
Therefore, we get 
	\begin{align*}
	0&\leq\limsup_{M \to \infty} \limsup_{H \to \infty}\vert s_1(H,M) \vert 
	\\ & \leq \limsup_{M \to \infty} \limsup_{H \to \infty}  \sum\limits_{A \in [-H,H[^d_I \cap \bigcup_{M<p\in \mathcal{P}} U_p} \frac{\ell_{A,H}}{(2H)^d}  \\
	& \leq \limsup_{M \to \infty} \limsup_{H \to \infty}  \sum\limits_{A \in [-H,H[^d_I \cap \bigcup_{M <p\in \mathcal{P}} U_p}  \frac{c}{(2H)^d} \\
	& = c \limsup_{M \to \infty} \limsup_{H \to \infty} \   \displaystyle{ \frac{\vert  [-H,H[_I^d \cap \bigcup_{M<p\in \mathcal{P} } U_p \vert}{(2H)^d} } \\
	&=   c \limsup_{M \to \infty} \ \overline{\rho}\left( \bigcup\limits_{M<p\in \mathcal{P}} U_p\right) =0,
	\end{align*}
	where the last equality follows from  Condition \eqref{densitycond}. 
	Using \eqref{newcond2} and \eqref{newcond3} we get
	\begin{align*}
	0 \leq \limsup_{M \to \infty} \limsup_{H \to \infty} \vert s_2(H,M) \vert
	&= \limsup_{M \to \infty} \limsup_{H \to \infty} \sum_{p \in \mathcal{P}, \ M <p < H^\alpha} \frac{\vert U_p \cap [-H, H[_I^d \vert}{(2H)^d} \\
	&\leq \limsup_{M \to \infty} \limsup_{H \to \infty} \sum_{p \in \mathcal{P},\  M <p < H^\alpha}  v_p 
	= 0. 
	\end{align*}
	For $\alpha=0$ on the other hand, we have for $M>1$ that $s_1(H,M)=0=s_2(H,M)$.
	Using Corollary \ref{denofU} we get
	\begin{align*}
	\lim_{M \rightarrow \infty} \lim_{H \rightarrow \infty} s_3(H,M) 
	& = \lim_{M \to \infty} \lim_{H \to \infty}  \sum\limits_{\nu \in M_\mathbb{Q}, \ \nu \leq M \text{ or } \nu=\infty} \displaystyle{\frac{ \mid [-H,H[^d_I \cap U_\nu  \mid }{ (2H)^d  }      } \\
	& = \lim_{M \to \infty} \sum\limits_{\nu \in M_\mathbb{Q}, \  \nu \leq M\  \text{ or} \  \nu= \infty} \rho(U_\nu \cap \mathbb{Z}^d) \\
	&= \lim_{M \to \infty}\sum\limits_{\nu \in M_\mathbb{Q}, \  \nu \leq M \ \text{or} \  \nu= \infty} s_{\nu} \\
	&= \sum_{\nu \in M_{\mathbb{Q}}} s_{\nu}.
	\end{align*}

 \end{proof}
A natural question is whether some of the conditions in Theorem \ref{addendum} are redundant. This is not the case as the next two examples show.
\begin{example}\label{ex:nocond1}
In this example, we construct $(U_\nu)_{\nu \in M_\mathbb{Q}}$ such that Conditions \eqref{newcond} and \eqref{densitycond} are verified but  the conclusion of  Theorem \ref{addendum} does not hold.	As in Example \ref{st_ex}, we denote by $p_j$ the $j$th prime. We choose
	$$ U_\infty= \emptyset, \quad U_{p_j} = \{ p_j, p_{j+1}, \dots , p_{2^j}\}. $$
	The same argument as in Example \ref{st_ex} shows that $\mu_{p_j}(\partial (U_{p_j}))=0$. 
	By the prime number theorem we have that $\mathcal{P}$ has density zero and thus our choice satisfies Condition \eqref{densitycond} even without taking the limit in $M$. 

	Note that for $p_j >H>0$ we have $U_{p_j} \cap [-H, H[= \emptyset$ and thus Condition \eqref{newcond} is satisfied with $c=1=m$. Now we will show that the conclusion of the theorem does not hold. First we show that $\mu = \infty$, if the limit exists. Note that for this example $[-H,H[_I = [-H,H[$ for all $H>0$. Let $L\in \mathbb{N}$, then we compute 
	\begin{align*}
	&\sum_{A\in [-p_{2^L}, p_{2^L}[}  \frac{ \vert \{ \nu \in M_\mathbb{Q} \ \vert \ A \in U_\nu \} \vert}{2 p_{2^L}}
	= \sum_{j=1}^{2^L} \frac{\vert U_{p_j} \cap [-p_{2^L}, p_{2^L}[ \vert}{2 p_{2^L}} \\
	&= \sum_{j=1}^{L-1} \frac{\vert \{p_j, \dots, p_{2^j} \} \cap [-p_{2^L}, p_{2^L}[ \vert }{2p_{2^L}}
	+ \sum_{j=L}^{2^L-1} \frac{\vert \{p_j, \dots, p_{2^L-1}\} \vert}{2p_{2^L}} \\
&	= \sum_{j=1}^{L-1} \frac{2^j-j+1}{2p_{2^L}} + \sum_{j=L}^{2^L-1} \frac{2^L-j}{2p_{2^L}} \\
	&= \frac{1}{2 p_{2^L}} \left[ \left( 2^L-2 \right) - \frac{L(L-1)}{2}+ (L-1) + 2^L(2^L- L) - \frac{(2^L-1)2^L}{2} +\frac{L(L-1)}{2} \right] .
	\end{align*}
Hence, for $L$ sufficiently large we get
	\begin{equation} \label{lower bound}
	\sum_{A\in [-p_{2^L}, p_{2^L}[}  \frac{ \vert \{ \nu \in M_\mathbb{Q} \ \vert \ A \in U_\nu \} \vert}{2 p_{2^L}}
	\geq \frac{1}{5 p_{2^L}} 2^{2L}.
	\end{equation}
	Thus, using the prime number theorem, we obtain
	\begin{align*}
	\limsup_{H\rightarrow \infty} \sum_{A\in [-H, H[_I} \frac{ \vert \{ \nu \in M_\mathbb{Q} \ \vert \ A \in U_\nu \} \vert}{(2H)}
	&\geq \limsup_{L\rightarrow \infty} \sum_{A\in [-p_{2^L}, p_{2^L}[} \frac{ \vert \{ \nu \in M_\mathbb{Q} \ \vert \ A \in U_\nu \} \vert}{2 p_{2^L}} \\
	&\geq \limsup_{L\rightarrow \infty} \frac{1}{5 p_{2^L}} 2^{2L} \\
	&= \limsup_{L\rightarrow \infty} \frac{2^L}{5 \ln(2) L  }
	= \infty.
	\end{align*}
	As noted above, all finite sets are null sets for the $p$-adic Haar measure. Thus, we have $s_p = 0$ for all $p\in \mathcal{P}$ and $s_\infty = \frac{1}{2} \mu_\infty(\emptyset)=0$ and hence
	$$ \sum_{\nu \in M_{\mathbb{Q}}} s_\nu =0. $$
	This means the conclusion of the theorem does not hold, if we only assume Conditions \eqref{densitycond} and \eqref{newcond}. 
\end{example}

\begin{example} \label{ex:nocond2}
	Next we construct an example that satisfies  Conditions \eqref{densitycond}, \eqref{newcond2}, \eqref{newcond3}, 
	and we show that the conclusion of the theorem does not hold. We set $U_\infty=\emptyset$ and for $p\in \mathcal{P}\setminus \{ p_{2^n} \ \mid \ n\in \mathbb{N} \}$ we define $U_p = \emptyset$. Inductively, we define the remaining $U_{p_{2^n}}$. We start with $U_{p_1}=\{1\}$. If $U_{p_{2^n}}=\{m^2\}$  for $m\in \mathbb{N}$, then we define
	\begin{align*}
	U_{p_{2^{n+1}}} = \begin{cases}
	\{m^2 \},& \text{if } \vert \{ j\in \mathbb{N} \ \mid \ j\leq n, U_{p_{2^j}}=\{m^2 \} \} \vert<m^3, \\
	\{(m+1)^2\},& \text{else}.
	\end{cases}
	\end{align*}
	The same argument as for Example \ref{st_ex} applies here and 
	gives $\mu_{p_j}(\partial (U_{p_j})) =0$. We compute 
	\begin{align*}
	\overline{\rho}\left( \bigcup_{\nu\in M_\mathbb{Q}} U_\nu\right)
	= \overline{\rho}\left( \{ n^2 \ : \ n\in \mathbb{N} \} \right) 
	&= \limsup_{H\rightarrow \infty}  \frac{\vert \{ n\in \mathbb{N} \ : \ n^2 < H \} \vert}{2H} \\
	& \leq \limsup_{H\rightarrow \infty} \frac{\sqrt{H}}{2 H} 
	=0.
	\end{align*}
	Hence, Condition \eqref{densitycond} is satisfied. We have for $H>p_n$
	\begin{align*}
	\vert [-H, H[ \cap U_{p_n} \vert
	= \begin{cases}
	1 & \text{if } n= 2^k \text{ for some }  k\in \mathbb{N},\\
	0 & \text{otherwise}.
	\end{cases}
	\end{align*}
	Thus, we may pick
	\begin{align*}
	v_{p_n} = \begin{cases}
	\frac{1}{p_n} &  \text{if } n= 2^k \ \text{ for some } k\in \mathbb{N},\\
	0 & \text{otherwise}.
	\end{cases}
	\end{align*}
	By the prime number theorem, there exists some constant $D>0$, 
	such that
	\begin{align*}
	\sum_{p_n \in \mathcal{P}} v_{p_n} 
	= \sum_{k \geq 1} \frac{1}{p_{2^k}} 
	= \sum_{k \geq 1} \frac{2^k \ln(2^k)}{p_{2^k}} \frac{1}{2^k \ln(2^k)}
	\leq \sum_{k \geq 1} \frac{D}{2^k \ln(2^k)} < \infty.
	\end{align*}
	Thus, conditions \eqref{newcond2} and \eqref{newcond3} are satisfied. \\
	By construction, we have
	\begin{align*}
	 \vert \{ \nu \in M_\mathbb{Q} \ \vert \ A\in U_\nu \} \vert
	= \begin{cases}
	m^3  &A= m^2 \text{ for some } m\in \mathbb{N}_{>0},\\
	0  &\text{else}.
	\end{cases}
	\end{align*}
	Therefore, we get
	
	\begin{align*}
	\sum_{A\in [-H, H[_I} \frac{ \vert \{ \nu \in M_\mathbb{Q} \ \vert \ A\in U_\nu \} \vert}{(2H)}
	& \geq \sum_{m=1}^{\lfloor \sqrt{H} \rfloor - 1} \frac{m^3}{2H}
	= \frac{1}{2H} \frac{\lfloor \sqrt{H} \rfloor^2 }{4} (\lfloor  \sqrt{H} \rfloor -1)^2 \\ &
	\geq \frac{1}{32} (\lfloor \sqrt{H} \rfloor -1)^2.
	\end{align*}
	Hence,
	\begin{align*}
	\limsup_{H\rightarrow \infty} \sum_{A\in [-H, H[_I} \frac{ \vert \{ \nu \in M_\mathbb{Q} \ \vert \ A\in U_\nu \} \vert }{(2H)} = \infty.
	\end{align*}
	On the other hand, by the same argument as in the previous example we obtain
	\begin{align*}
	\sum_{\nu \in M_{\mathbb{Q}}} s_\nu =0.
	\end{align*}
\end{example}

\section{Applications}\label{sec4}

We can apply Theorem \ref{addendum} to sets, whose densities were computed via the local to global principle of Theorem \ref{poonen} and fulfill Conditions \eqref{newcond}, \eqref{newcond2} and \eqref{newcond3}.


For example we can compute the expected number of common prime divisors of all basic   minors of a  rectangular non-unimodular matrix. The rigorous statement reads as follows.

\begin{corollary}
Let $n<m$ be positive integers, 
and let us denote by $R$ the set of rectangular unimodular matrices in $\mathbb{Z}^{n \times m}$. Then the corresponding system $(U_\nu)_{\nu \in M_\mathbb{Q}}$ is given as in \cite{densitiesunimodular}, i.e.,
$U_\infty = \emptyset$ and for $p\in \mathcal{P}$ denote by $U_p$ the set of all matrices in $\mathbb{Z}_p^{n\times m}$ whose $n$-minors are all divisible by $p$.

 Then the expected value of the system $(U_\nu)_{\nu \in M_\mathbb{Q}}$ exists and is given by
\begin{align}\label{murec}
\mu &= \sum_{\nu \in M_\mathbb{Q}} s_\nu = \sum\limits_{p   \in \mathcal{P}} \left(1- \prod\limits_{i=0}^{n-1} \left(1-\frac{1}{p^{m-i}} \right)\right). 
\end{align}
And the average number of primes that divide all $n$-minors of a rectangular non-unimodular matrix is given by
\begin{equation}\label{restmurec}
\mu_{R^C} = \displaystyle{\frac{\mu}{1-\prod\limits_{i=0}^{n-1}\frac{1}{\zeta(m-i)} }}, 
\end{equation}
where $\zeta$ denotes the Riemann zeta function.
\end{corollary}
\begin{proof}
Recall from \cite{densitiesunimodular}, that all conditions of Theorem \ref{poonen} are satisfied for the corresponding system $(U_\nu)_{\nu \in M_\mathbb{Q}}$, that for $p \in \mathcal{P}$ we have that $$s_p = \left(1- \prod\limits_{i=0}^{n-1}\left(1-\frac{1}{p^{m-i}}\right)\right)$$ and thus $$\rho(R)=\rho(P^{-1}(\{\emptyset\}))= \prod\limits_{i=0}^{n-1}\frac{1}{\zeta(m-i)} .$$

Thanks to Lemma \ref{passtoT} we are left with proving that the additional assumptions on the system $(U_\nu)_{\nu \in M_\mathbb{Q}}$ of Theorem \ref{addendum} are satisfied. 

For this, let us choose $\alpha=1$ and denote by $f_i$ the function associating to $A\in \mathbb{Z}^{n\times m}$ some fixed $n$-minor. Then we have for all $A\in [-H, H[^{n\times m}$ the inequality $ f_i(A)  \leq (2H)^n$ for all $i$. We exclude that $f_i(A)$ vanishes for all $i$, since then we land in $I$. 

Recall, that
$$\ell_{A,H} = \left\vert \left\{ p \in \mathcal{P} \mid p>H, \ A \in U_p \cap [-H,H[^{nm}_I \right\} \right\vert.$$
Thus, we have that 
$$H^{\ell_{A,H}} \leq \prod\limits_{\substack{p \in \mathcal{P} \\ p>H, \ A \in U_p \cap [-H,H[^{nm}_I}} p.$$
Further, observe that 
		$$\prod\limits_{\substack{p \in \mathcal{P} \\ p>H, \ A \in U_p \cap [-H,H[^{nm}_I}} p = \gcd( (f_i(A))_i) \leq (2H)^n.$$
		Hence, we get that $H^{\ell_{A,H}} \leq (2H)^n$, and thus Condition \eqref{newcond} is satisfied.

		To verify Condition \eqref{newcond2} we want to show that there exists a sequence $(v_p)_{p  \in \mathcal{P}}$ such that for all $p<H$ 
		we have that $ \mid U_p \cap [-H,H[^{nm}_I \mid \leq  v_p (2H)^{nm}$. The set of non-full rank matrices over $\mathbb{F}_p$ has size 
		
		$$p^{nm}-\prod\limits_{i=0}^{n-1} (p^m-p^i) \leq 2^n p^{m(n-1)+n-1} = 2^n p^{(m+1)(n-1)}.$$ 
		We can fix one non-full rank $n \times m$ matrix over $\mathbb{F}_p$, for which we have less than or equal to $2^np^{(m+1)(n-1)}$ choices. 
		For 
		this fixed matrix there are less than or equal to $(\lceil \frac{2H}{p} \rceil)^{nm}$ lifts to $\mathbb{Z}^{n\times m} \cap [-H,H[^{nm}$. Hence, we have for $p<H$
		\begin{align*} 
		\mid U_p \cap [-H,H[^{nm}_I \mid & \leq 2^n \left(\left\lceil \frac{2H}{p}\right\rceil\right)^{nm} p^{(m+1)(n-1)}  \\
		& \leq  2^n\left( \frac{2H}{p} +1\right)^{nm} p^{(m+1)(n-1)} \\
		& \leq  2^n(3H)^{nm}p^{nm -m+n-1-nm} \\
		& \leq  6^{nm}H^{nm} \frac{1}{p^2}.
		\end{align*}
		Thus $(v_p)_{p  \in \mathcal{P}}$ can be chosen  to be $(\frac{6^{nm}}{p^2})_{p\in \mathcal{P}}$, which also satisfies Condition \eqref{newcond3}. Hence, \eqref{murec} follows and Lemma \ref{passtoT} implies \eqref{restmurec}.
\end{proof}

By choosing $n=1$ we get the mean  of numbers of primes dividing non-coprime $m$-tuples of integers and of course choosing $n=1$ and $m=2$ will give the mean  
of numbers of primes dividing non-coprime pairs of integers.
 
The results of  \cite{primenumbereisenstein} regarding the average amount of primes for which a  non-monic  Eisenstein-polynomial {satisifies the criterion of Eisenstein} follow immediately as corollary using Theorem \ref{addendum}.

\begin{corollary}\label{cor:eisenstein}
Let $d\geq 2$ be an integer. The expected number of primes $p$ for which an Eisenstein polynomial of degree $d$ is $p$-Eisenstein, is given by
\[\left(1-\prod_{p\in \mathcal P}\left(1-\frac{(p-1)^2}{p^{d+2}}\right)\right)^{-1}\sum_{p\in \mathcal P}\frac{(p-1)^2}{p^{d+2}}.\]
\end{corollary}
\begin{proof}
In Theorem  \ref{addendum} simply use the system $U_p=(p\vZ_p\setminus p^2\vZ_p)\times p\vZ_p^{d-1} \times (\vZ_p\setminus p\vZ_p)$ and choose $\alpha=1$.
Let $E(H)$ be the set of Eisenstein polynomials of height at most $H$.
Condition \eqref{newcond} is trivially verified, as no polynomial can be Eisenstein with respect to a prime larger than its height.
Condition \eqref{newcond3} is verified thanks to the rough estimate $|E(H)|\leq \lceil 2H/p\rceil^{d}\cdot H$, and this is enough for our purposes because $\sum_{p\in \mathcal{P}} 1/ p^{d}$ converges for $d\geq 2$.

\end{proof}

Moreover, using the system $\overline E_p$'s given in \cite{micheli2016densityeis}  we can obtain the expected number of primes for which a given polynomial $f(x)$ is Eisenstein for some shift $f(x+i)$.
For the sake of completeness and to show how easy it is to apply Theorem \ref{addendum} when the system of $U_p$'s is given, let us now compute this expected value.
Let $d\geq 3$ be a positive integer and let $\overline E_p$ be the set of polynomials $f$ of degree $d$ in $\vZ_p[x]$ such that there exists $i$ 
for which $f(x+i)$ is $p-$Eisenstein in $\vZ_p[x]$.
Cleary, the expected value of the system  $(U_p)_{p \in \mathcal{P}}=(\overline E_p)_{p \in \mathcal{P}}$ (where every $\overline E_p$ is seen as a subset of $\vZ_p^{d+1}$)  is exactly the number we are interested in.
Let us verify that the system satisfies the three conditions
\begin{itemize}
\item Condition \eqref{densitycond} has already been verified in \cite{micheli2016densityeis}, as it is needed to compute the density of the set of shifted Eisenstein polynomials.
\item Condition \eqref{newcond} is immediate with the choice $\alpha=1$. In fact, the 
primes $p$, for which a polynomial is $p$-Eisenstein, divide the discriminant of the
 polynomial. 
Notice that $f(x+i)$ has the same discriminant as 
$f(x)$, which is bounded by $CH^{2d-2}$ (as the discriminant is homogeneous of degree $2d-2$), for some absolute constant $C$. Thus, one obtains that the product of the primes $p$ larger than $H$ for which $f(x+i)$ is $p$-Eisenstein for some $i$ is bounded by an absolute constant.
\item Condition \eqref{newcond3} is also easy to verify. First, observe that if $f(x+i)$ is $p$-Eisenstein for some $i\in\vN$, then $i$ can be chosen less than $p$ (see for example \cite[Lemma 6]{micheli2016densityeis}).
So that the set of shifted $p$-Eisenstein polynomials $S$ in $[-H,H[^{d+1}$ is absolutely bounded by $|E(H)|\cdot p$, where $E(H)$ is the set of $p$-Eisenstein polynomials of degree $d+1$ and height at most $H$.
Finally, one can very roughly give the estimate $|E(H)|\leq \lceil 2H/p\rceil^{d}\cdot H$, which is anyway enough for our purposes as $\sum_{p\in \mathcal{P}} 1/ p^{d-1}$ converges for $d\geq 3$.
\end{itemize}

Using the computation of the $p$-adic measure of $\overline E_p$ from \cite{micheli2016densityeis} and Lemma \ref{passtoT} (to obtain the restricted mean), one obtains the following result:

\begin{corollary}
For $d\geq 3$, the mean of the system $(\overline E_p)_{p \in \mathcal{P}}$ (i.e., the expected number of primes for which a polynomial $f(x)$ is Eisenstein after some shift) is 
\[\sum\limits_{{p\in \mathcal{P}} }\frac{(p-1)^2}{p^{d+1}}.\]
The expected value restricted to the set of shifted Eisenstein polynomials is 
\[\left(1-\prod\limits_{{p\in \mathcal{P}} }\left(1-\frac{(p-1)^2}{p^{d+1}}\right)\right)^{-1}\sum\limits_{{p\in \mathcal{P}} }\frac{(p-1)^2}{p^{d+1}}.\]
\end{corollary}

\section*{Acknowledgments}
The second author is thankful to the Swiss National Science Foundation grant number  20020\_172623.
The third author is  partially supported by Swiss National Science Foundation grant number 188430.

\bibliographystyle{plain}
\bibliography{biblio}

\end{document}